\pdfoutput=1
\RequirePackage{ifpdf}
\ifpdf 
\documentclass[pdftex]{sigma}
\else
\documentclass{sigma}
\fi

\usepackage{bm}

\newcommand{\R}{\mathbb{R}}

\newcommand{\C}{\mathbb{C}}

\newcommand{\MV}{{\mathcal V}}

\newcommand{\MF}{{\mathcal F}}

\newcommand{\MM}{{\mathcal M}}
\newcommand{\MS}{{\mathcal S}}

\renewcommand{\emph}[1]{{\itshape{#1}}}
\newcommand{\ii}{{\mathrm{i}}}
\let\Re\relax
\DeclareMathOperator{\Re}{Re}

\renewcommand{\epsilon}{\varepsilon}

\DeclareMathOperator{\Area}{Area}

\let\Re\relax
\DeclareMathOperator{\Re}{{\mathrm Re}}

\newcommand{\pd}{\partial}

\newcommand{\inner}[2]{{\left\langle{#1}, {#2}\right\rangle}}
\newcommand{\nai}{\langle}
\newcommand{\seki}{\rangle}

\numberwithin{equation}{section}
\newtheorem{Theorem}{Theorem}[section]
\newtheorem{Corollary}[Theorem]{Corollary}
\newtheorem{Lemma}[Theorem]{Lemma}

{ \theoremstyle{definition}
\newtheorem{Definition}[Theorem]{Definition}

\newtheorem{Remark}[Theorem]{Remark} }

\begin{document}
\allowdisplaybreaks

\renewcommand{\thefootnote}{}

\newcommand{\arXivNumber}{2307.08537}

\renewcommand{\PaperNumber}{034}

\FirstPageHeading

\ShortArticleName{A Weierstrass Representation Formula for Discrete Harmonic Surfaces}

\ArticleName{A Weierstrass Representation Formula \\ for Discrete Harmonic Surfaces\footnote{This paper is a~contribution to the Special Issue on Differential Geometry Inspired by Mathematical Physics in honor of Jean-Pierre Bourguignon for his 75th birthday. The~full collection is available at \href{https://www.emis.de/journals/SIGMA/Bourguignon.html}{https://www.emis.de/journals/SIGMA/Bourguignon.html}}}

\Author{Motoko KOTANI~$^{\rm a}$ and Hisashi NAITO~$^{\rm b}$}
\AuthorNameForHeading{M.~Kotani and H.~Naito}

\Address{$^{\rm a)}$~The Advanced Institute for Materials Research (AIMR), Tohoku University, Japan}
\EmailD{\href{mailto:motoko.kotani.d3@tohoku.ac.jp}{motoko.kotani.d3@tohoku.ac.jp}}

\Address{$^{\rm b)}$~Graduate School of Mathematics, Nagoya University, Japan}
\EmailD{\href{mailto:naito@math.nagoya-u.ac.jp}{naito@math.nagoya-u.ac.jp}}

\ArticleDates{Received July 17, 2023, in final form April 12, 2024; Published online April 17, 2024}

\Abstract{A discrete harmonic surface is a trivalent graph which satisfies the balancing condition in the $3$-dimensional Euclidean space and achieves energy minimizing under local deformations. Given a topological trivalent graph, a holomorphic function, and an associated discrete holomorphic quadratic form, a version of the Weierstrass representation formula for discrete harmonic surfaces in the $3$-dimensional Euclidean space is proposed. By using the formula, a smooth converging sequence of discrete harmonic surfaces is constructed, and its limit is a classical minimal surface defined with the same holomorphic data. As an application, we have a discrete approximation of the Enneper surface.}

\Keywords{discrete harmonic surfaces; minimal surfaces; Weierstrass representation formula}

\Classification{53A70; 53A10; 52C26}

\begin{flushright}
\begin{minipage}{75mm}
\it On the occasion of Jean-Pierre Bourguignon'\\ 75th birthday
 \end{minipage}
\end{flushright}

\renewcommand{\thefootnote}{\arabic{footnote}}
\setcounter{footnote}{0}

\section{Introduction}
The purpose of the present paper is to construct a discrete version of the well known
{\em Weierstrass representation formula} for a minimal surface and to show its applications.
A minimal surface is defined as a minimizing surface of the area functional under local variations
and a central research object in geometric analysis.
Minimal surfaces are seen everywhere in nature and also used in engineering products
because they enjoy both effectiveness and natural beauty at the same time.

Recent years, various notions of {\em discrete surfaces} have been proposed.
The {\em triangularizations/polygonalizations} of a topological surface have been traditionally used
both in pure and applied mathematics and been proved useful in history.
Discretization of differential geometric surface in geometric analysis by U.~Pinkall and K.~Polthier
\cite{bobenko-pinkall, pinkall-polthier}, of integrable systems lead by A.~Bobenko et al.\
\cite{bobenko-bucking-sechelmann, bobenko-hoffmann, bobenko-pottmann-wallner}
are among those challenges.
The present paper is based on the {\em discrete surface theory} introduced in \cite{kotani-naito-omori}.
The notion of discrete surface to be harmonic and the {\em balancing condition} are proposed there.
The condition for discrete surface to be minimal under the area variational formula is also given.
It should be noted that isometry is a too rigid property in discrete surface theory while those are important in classical surface theory.
Surfaces with harmonicity are much richer than minimal surfaces,
area minimizing surfaces among isometric surfaces in the Euclidean space, to be studied in the discrete case.
We call discrete surfaces in the Euclidean space minimizing the energy functional,
and satisfying the balancing condition as a consequence {\em discrete harmonic surfaces}.

The convergence of a sequence of subdivided discrete surfaces of a given discrete surface is studied in \cite{kotani-naito-tao, tao}.
Although we proved the energy monotonicity formula in the convergence,
the limit {\em surface} has singularities in general.
We found there the balancing condition plays an important role to control the regularity of the limit surface.
As an application of the Weierstrass representation formula is to establish a method to analyze the singularities for a good class of discrete surfaces.
We expect we should have a better control in the case of minimal surfaces or harmonic surfaces which has a ``Weierstrass representation formula'' given by holomorphic data.

In order to establish a kind of {\em Weierstrass representation formula} for discrete harmonic surfaces,
we need a notion of {\em conformal} structures of a discrete surface,
and holomorphic quadratic differentials with respect of the conformal structure.
A notion of conformal structure for triangularizations of a topological surface is introduced by
W.~Thurston \cite{thurston} motivated in discrete approximation of the Riemannian mapping between complex surfaces.
He noticed a map between two surfaces preserving their circle packings can be used as a discrete conformal map
because they hold circles and contact angles.
The notion develops discrete Riemann mapping theorem \cite{rodin-sullivan} between complex surfaces
and the discrete complex function theory (cf.\ \cite{circle-packing}).
On the other hand, the notion of holomorphicity and holomorphic differential were introduced to study discrete surfaces
by U.~Pinkall and W.Y.~Lam in \cite{lam-pinkall} and have been used for the study of minimal surfaces by \cite{lam, lam2}.

By adapting those notions to discrete harmonic surfaces in our setting,
we are successful to construct a Weierstrass representation formula for discrete harmonic surfaces with the data of their Gauss map and holomorphic quadratic differentials and construct a converging sequence of discrete harmonic surfaces, which are obtained by iterating subdivision of a given discrete harmonic surface.

\begin{theorem*}[Weierstrass representation formula for a discrete harmonic surface, Theorem~\ref{claim:weierstrass}]
 Given a topological planer trivalent graph $M=(M,V,E)$, where $V$ is the set of vertices, and~$E$ is the set of edges.
 With respect to a complex coordinate $z\colon M \to \C$ for a holomorphic function~$g$ on~$\C$
 and a holomorphic quadratic differential $q$ on $M$ associated with $g$,
 let us define a~map~${X\colon M \to \R^3}$
 \begin{displaymath}
 X = \int \Re\, {\rm d}F,
 \end{displaymath}
 with
 \begin{align}
 {\rm d}F(e_{ab}) &= \frac{\ii q(e_{ab})}{{\rm d}g(e_{ab})}(1-g_ag_b, \ii(1+g_ag_b), g_a+g_b)\nonumber\\
 &=-\overline{\tau}(e_{ab}) (1-g_ag_b, \ii(1+g_ag_b), g_a+g_b), \qquad e_{ab} \in E,\label{eq:weierstrass}
 \end{align}
which connects the vertices $v_a$ and $v_b$ in $V$. It defines a discrete harmonic surface in $\R^3$.
 The above equation is called Weierstrass representation formula for a discrete harmonic surface.
\end{theorem*}

As an application of the theorem, we have
\begin{theorem*}[Theorem \ref{homothety} and Corollary \ref{claim:homothety:cor}]
 A smoothly converging sequence of discrete harmonic surfaces $X_n \colon M_n \to \R^3$
 with a given initial discrete harmonic surface $X_0\colon M_0 \to \R^3$ with the holomorphic data $g$ and $q$ is constructed,
 where $M_n$ are iterated subdivisions of $M_0$ in $\C$.
 The limit is a classical minimal surface given by the classical Weierstrass representation with the holomorphic data.
\end{theorem*}
The limit surface is considered as a hidden continuous object in a given discrete surface.
We expect the method is useful to draw minimal surfaces in computer graphics.
Actually, in the last section, we show as an example a sequence of discrete harmonic surfaces which converges to the Enneper surface, a famous classical minimal surface.

\section{Minimal surfaces in differential geometry}

In geometric analysis,
minimal surfaces have been studied intensively with variational methods.
A minimal surface is a solution of the Euler--Lagrange equation to the area functional.
To be more precise,
for a given isometric immersion $X\colon M \to \R^3$ of a smooth surface $M$ with the unit normal vector field $N\colon M \to S^2$,
consider its smooth deformation
\begin{gather*}
 X_t = X +tN \qquad \text{for}\quad t \in (-\epsilon, \epsilon).
\end{gather*}
Because the {\em area variation formula} is expressed as for $A(t)= \Area(X_t)$, we have the {\em Steiner formula}\label{steinerformula}
\begin{gather*}
 {\rm d}A(t) = \big(1 -2tH+ 4t^2 K +O\big(t^3\big)\big){\rm d}A
\end{gather*}
for a local variation $X_t$ of a given surface $X_0$ with $H$ and $K$ as
its {\em mean curvature} and {\em Gauss curvature}, a minimal surface is characterized to be a surface with the mean curvature $H =0$.

It is known that a differential surface admits an isothermal coordinate $(u,v)$ with which the first fundamental from is expressed as
\begin{gather*}
 I ={\rm e}^{2\lambda}\big({\rm d}u^2+ {\rm d}v^2\big).
\end{gather*}
With respect the complex coordinate $z=u+\ii v$; the Laplacian is given as
\begin{gather*}
 \Delta = {\rm e}^{-2\lambda} \left( \frac{\pd^2}{\pd u^2} + \frac{\pd^2}{\pd v^2}\right).
\end{gather*}
The structure equations for a smooth surface $X\colon M \to \R^3$ with $N$ the unit normal vector field
and $Q = \nai X_{zz}, N \seki$ are given
\begin{Theorem}[structure equations for a smooth surface]
 \[
 X_{z \overline{z}}= \frac{1}{4} {\rm e}^{2\lambda}H N, \qquad
 X{zz} = 2 \lambda_{z}X_z + QN, \qquad
 N_z = -2 {\rm e}^{-2\lambda}Q X_{\overline{z}} - \frac{1}{2} HX_z,
 \]
 where $H$ is the mean curvature of $X$.
 The integrability conditions are given for $Q$, $H$ and $\lambda$
 \[
 \lambda_{z\overline{z}}=-{\rm e}^{-2\lambda} Q \overline{Q} + \frac{1}{16} {\rm e}^{2\lambda}H^2 =0, \qquad
 4Q_{\overline{z}} = {\rm e}^{2\lambda} H_z.
 \]
\end{Theorem}
When $M$ is a minimal surface, the above equations are given
\[
 X_{z\overline{z}}=0,\qquad
 X_{zz} = 2\lambda_z X_z + QN,\qquad
 N_z = -2{\rm e}^{-2\lambda}Q X_{\overline{z}}.
\]
That indicates $X$ is a vector valued harmonic function and
$X_z$ is a vector valued holomorphic function satisfying
\begin{displaymath}
 \nai X_z, X_z \seki = 0.
\end{displaymath}
Additionally, $Q{\rm d}z{\rm d}z$ is a holomorphic quadratic differential,
due to $X_{z\overline{z}}=0$.

The Weierstrass representation for a minimal surface is given as
\begin{Theorem}[Weierstarss representation formula for a minimal surface]
 With a holomorphic quadratic differential $Q$ and a meromorphic function $g$, let us define
 a vector valued function~${X\colon M \to \R^3}$
 \begin{displaymath}
 X = \Re \int \frac{Q}{g_z}{\rm d}z
 \left(
 \frac{1}{2}\big(1-g^2\big), \frac{\ii}{2}\big(1+g^2\big), g
 \right)
 \in \R^3.
 \end{displaymath}
 It is a minimal surface with the Gauss map
 \begin{displaymath}
 N = \frac{1}{1+|g|^2}\left(g+\overline{g}, \ii (\overline{g}-g), |g|^2-1 \right) \in S^2.
 \end{displaymath}
\end{Theorem}

Moreover, the 1-parameter family
\begin{displaymath}
 X(\theta) = \Re \int {\rm e}^{\ii \theta}\frac{Q}{g_z}{\rm d}z
 \left(
 \frac{1}{2}\big(1-g^2\big), \frac{\ii}{2}\big(1+g^2\big), g
 \right)
 \in \R^3
\end{displaymath}
is proved to be a family of minimal surfaces and called an {\em associated family} of $X_0$.
The minimal surface $X(\pi/2)$ is called the conjugate minimal surface of $X_0$.

\section{Basic knowledge of discrete surfaces}
Let us review some results on discrete surfaces in \cite{kotani-naito-omori, kotani-naito-tao, tao} that we used in the present paper.

\begin{Definition}[discrete surface \cite{kotani-naito-omori}]
 Consider a trivalent graph $M=(M, V, E)$ with the set~${V=V(M)}$ of vertices and the set $E=E(M)$ of oriented edges of $M$,
 and a non-degenerate vector field $N\colon M \to S^2$.
 Namely, $N$ takes different values for vertices next to each other.
 Its realization $X\colon M \to \R^3$ is called a {\em discrete surface} in $\R^3$ with $N$ as its Gauss map,
 when it has a non-degenerate corresponding tangent vector space of each vertex determined
 by the three nearest vertices of the vertex whose unit normal vector is equals to $N$.
\end{Definition}

Because a discrete surface does not have a Riemannian metric,
we define the first fundamental form and the second fundamental form at each vertex is
the weighted average over those of triangles surrounding the vertex. To be more precise,
for each vertex $v \in V$, let us denote the set $E_v$ of edges emerging at $v$
\begin{gather*}
 E_v = \{ e \in E \mid o(e ) =v \},
\end{gather*}
and the adjacency vertices of $v$, $v_1$, $v_2$, and $v_3$.
Then in $\R^3$
\begin{gather*}
 X_1 = X (v_1),
 \qquad
 X_2 = X (v_2),
 \qquad
 X_3 = X (v_3)
\end{gather*}
are the adjacency vertices of $X_0= X(v_0)$,
and the unit normal vectors at $X_0$, $X_1$, $X_2$, and $X_3$ are denoted by $N_0$, $N_1$, $N_2$ and $N_3$, respectively.

Consider the triangle $\triangle_{0,a,b}$ formed by the three vertices $X_0$, $X_a$, and $X_b$,
for $a,\,b \in \{1,2,3\}$ with $a\neq b$, and define its first and second fundamental formulas
 \begin{align*}
 & I_{ab}
=
 \begin{pmatrix}
 \inner{X_a-X_0}{X_a-X_0} & \inner{X_a-X_0}{X_b-X_0} \\
 \inner{X_b-X_0}{X_a-X_0} & \inner{X_b-X_0}{X_b-X_0}
 \end{pmatrix},
 \\
 & I\!I_{ab}
 =
 \begin{pmatrix}
 \inner{X_a-X_0}{N_a-N_0} & \inner{X_a-X_0}{N_b-N_0} \\
 \inner{X_b-X_0}{N_a-N_0} & \inner{X_b-X_0}{N_b-N_0}
 \end{pmatrix},
 \end{align*}
and $H_{ab}$ and $K_{ab}$ by the trace of the matrix $I_{ab}^{-1}I\!I_{ab}$ as are in the classical case.

Now we define the mean curvature $H(v_0)$ and the Gauss curvature $K(v_0)$ at $v_0$
by their weighted average by the areas of the triangles,
i.e.,
\[
 H(v_0) := \frac{1}{A(v_0)}\sum_{a,b} \sqrt{\det I_{ab}}H_{ab},
 \qquad
 K(v_0) := \frac{1}{A(v_0)}\sum_{a,b} \sqrt{\det I_{ab}}K_{ab},
 \]
where $A(v_0)$ is the sum of the area of the three triangles.
See \cite{kotani-naito-omori} for the details.

It should be noted that isometry of triangularizations is rigid,
and \cite{lam-pinkall} pointed out isometry is too strong requirement to be satisfied
to develop rich discrete surface theory. Actually, in our setting as well, we have the following.

\begin{Theorem}[conditions for discrete minimal surfaces, \cite{kotani-naito-omori}]
 The conditions for being a minimal surface is
 \begin{displaymath}
 \nai X_a-X_0, X_b- X_0\seki = \nai X_b-X_0, X_c- X_0\seki= \nai X_c-X_0, X_a- X_0\seki,
 \end{displaymath}
 where $a,b,c =1,2,3$.
\end{Theorem}
That indicates a local structure of minimal surfaces at each vertex is the identical,
i.e., trivalent with equal length and the same angle $2\pi/3$ between two of them.
The class is rather too rigid.

On the other hand, harmonic realizations of a discrete surface seem to be richer and useful to develop geometry.
\begin{Definition}[balancing condition and harmonic surface]
 When $X\colon M \to \R^3$ satisfies the {\em balancing conditions} at each vertex;
 \begin{equation}
 \label{eq:balancing}
 \sum_{e \in E_v}{\rm d}X(e) = (X_1-X_0) +(X_2-X_0) + (X_3-X_0) =0, \qquad v \in E,
 \end{equation}
 where $X_a = X (v_a)$ for $a=1,2,3$ are the adjacency vertices,
 we call $X$ a {\em harmonic realization} and $X(M)$ a {\em harmonic surface}.
\end{Definition}
We have proved a harmonic surface achieves minimizing of the energy function
\begin{displaymath}
 \mathbb{E} = \frac{1}{2}\sum_{u\sim v} \|X_u-X_v\|^2
\end{displaymath}
under local deformation \cite{kotani-naito-omori}.

\section[Construction of a sequence of iterated subdivisions of a given discrete surface]{Construction of a sequence of iterated subdivisions\\ of a given discrete surface}

One of the major interests of the discrete surface theory is to discover a continuum object hidden
in a discrete surface and study the relation between them.
A candidate of the continuum object is a converging limit,
if there exists, of a sequence of discrete surfaces obtained by iterated subdivisions of a discrete surface.

In \cite{kotani-naito-tao, tao}, for a given discrete surface, its subdivision is constructed.
Let us quickly review that.

Triangulations of a smooth surface has been useful in many aspect of mathematical study and it applications to problems in the real world.
In that case, a triangle lies on a plane or a~surface, and is subdivided by using the metric on it.
In our case, a discrete surface does not have a face which lies on a plane or a surface.
So the subdivision is not trivial.
We take two steps of the subdivisions process.
Namely, first we construct a topological subdivision and then realize it in $\R^3$ to have a subdivision of a discrete surface in $\R^3$.

A planar trivalent graph has a triangulation of a plane as its dual graph in the plane.
The triangulation is subdivided in the canonical way, and then take its dual to have a trivalent graph.
This is called the {\em Goldberg--Coxeter subdivision}.
We take it as a {\em topological subdivision},
and have a sequence of a topological discrete surface $M_n$ obtained by the iteration of this process for a~given discrete surface $M_0$.
\begin{figure}[hbtp]
 \centering
 \begin{tabular}{lll}
 a)
 &b)
 &c)
 \\
 \includegraphics[bb=0 0 133 133,width=100pt]{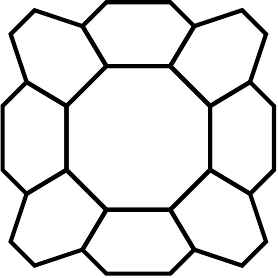}
 &\includegraphics[bb=0 0 133 133,width=100pt]{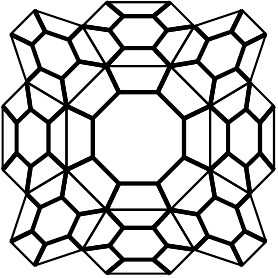}
 &\includegraphics[bb=0 0 133 133,width=100pt]{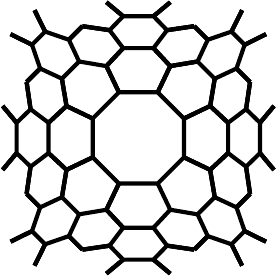}
 \end{tabular}
\caption{Goldberg--Coxeter subdivision.
 (a)~A part of trivalent graph consisting hexagons and an octagon. (b)~Take an $n$-gon inside each $n$-gon and then connect between the new vertex and original vertex. (c)~Remove original edges, we obtain the subdivision of~(a).}
\end{figure}

Now we take the second step to construct a sequence of discrete surfaces in $\R^3$.
For a given discrete surface $X_n$ in $\R^3$,
we take a part of it to have a planar graph $M_n$ in the plane,
and apply the Goldberg--Coxeter subdivision to have a new planar trivalent graph $M_{n+1}$.
We understand the process to be topological.
Next we realize $M_{n+1}$ in $\R^3$. When $X_n\colon M_n \to \R^3$ is given,
we consider it an a boundary condition of the equations whose solution should satisfy the balancing condition and to be $X_{n+1}$.

When the convergence of a sequence of iteratively subdivided discrete surfaces is discussed,
it is found that a naive subdivision obtained by solving the Dirichlet problem
with the original discrete surface as its the boundary condition does not work,
and its modification to make the subdivision to be a discrete harmonic,
i.e., that satisfies the balancing condition, is proved to be appropriate in the discussion of convergence \cite{tao}.

We have the following theorem for the convergence.

\begin{Theorem}[Tao \cite{tao}, Kotani--Naito--Tao \cite{kotani-naito-tao}]
 Let $\{X_{i} \}$ be a sequence of discrete harmonic surfaced iteratively constructed from $X_{0} = X$ as above.
 \begin{enumerate}\itemsep=0pt
 \item[$(1)$]
 The sequence $\{ X_n \}$ forms a Cauchy sequence in the Hausdorff topology
 and satisfies the energy monotonicity formula. 
 \item[$(2)$]
 The limit of the Cauchy sequence
 $\mathcal {M}_{\infty} = \overline{\bigcup M_i}$
 is divided into three kinds of sets:
 \begin{equation*}
 \mathcal {M}_{\infty} = \MM_{\MF} \cup \MM_{\MV} \cup \MM_{\MS}.
 \end{equation*}
 The first one $\MM_{\MF}$ is the set of accumulating points associated with each face in $M_i$.
 The second one is the set of all vertices, replaced as in the above step,
 i.e., \begin{math}
 \mathcal{M}_{\MV} = \bigcup_{i} M_i.
 \end{math}
 The third one $\MM_{\MS}$ is the set of the rest of the accumulating points.
 We know little about $\MM_{\MS}$ in general,
 however, we prove an un-branched discrete surface do not have such $\MM_{\MS}$.
 \end{enumerate}
\end{Theorem}
The regularity of the limit set is not trivial at all,
although we have the energy monotonicity formula.
That gives a motivation of the present paper.
When we have a representation formula with holomorphic data, we can expect to develop finer analysis of singularities.
\section{Discrete holomorphic quadratic differentials}
Let $M=(M, V, E)$ be a discrete surface with the set $V$ of vertices and the set $E$ of oriented edges of $M$.
By definition, $M$ is a trivalent graph. It has a non-degenerate unit normal vector field $N\colon V \to S^2$,
namely, the normal vector $N_v$ at $v$ are different from the normal vector at any of its nearest neighboring vertices.

Let us introduce a {\it complex coordinate system}.
Namely, for each vertex $v$ in $V$, assign a~connected plane subgraph $U_v$ of $M$,
and a realization $z\colon U_v \to U$ from $U_v$ to a domain $U \subset \C$.

The notion of the holomorphic quadratic differentials is introduced in \cite{lam-pinkall} for a triangulation of a topological surface.
We imitate theirs to define a~holomorphic quadratic differential for a~discrete surface in our setting.

Let $q\colon E\to \C$ be a discrete function with a discrete $1$-form $\tau \colon E \to \C$,
i.e., $\tau(e_{ab}) = - \tau(e_{ba})$ for each oriented edge $e_{ab}$
which connects a vertex $v_a$ to a vertex $v_b$ and $e_{ba}$ its reverse edge, is uniquely defined by
\begin{gather*}
 q(e_{ab}) = \nai \tau(e_{ab}), \ii {\rm d}z(e_{ab}) \seki, \qquad
 0 = \nai \tau(e_{ab}), {\rm d}z(e_{ab}) \seki,
\end{gather*}
where ${\rm d}z(e_{ab}) = z(v_b)- z(v_a) \in \C$, and $\nai ~,~\seki$ is a complex innerproduct in $\C$.

From the definition, we see $\tau(e_{ab})= -\tau(e_{ba})$.
Actually $\tau$ is explicitly given as
\begin{gather*}
 q(e) = \ii \overline{\tau(e)}{\rm d}z(e).
\end{gather*}

\begin{Definition}[discrete holomorphic quadric differential]
 When a discrete function $q$ which satisfies
 \begin{gather*}
 \sum_{e \in E_v} q(e)=0 \qquad \text{for all}\quad v \in V
 \end{gather*}
 with the discrete $1$-form $\tau$ defined as above satisfies
 \begin{gather*}
 \sum_{e \in E_v} \tau(e)=0 \qquad \text{for all}\quad v \in V,
 \end{gather*}
 it is called a {\em holomorphic quadratic differential}.
\end{Definition}
It is easily checked that the notion is preserved under the M\"obius transformation in the coordinate $z$.
\begin{Remark}
 The holomorphic quadratic differential $q$ can be written by using a harmonic function
 with respect to the cotangent Laplacian $\Delta_{\cot}$ (cf.\ \cite{pinkall-polthier}).
 More precisely, for a function~${u \colon V \to \R}$,
 define a $\C$-valued function $\lambda_{abc} $ for every triple $(a,b,c)$ such that
 \begin{gather*}
 \nai \lambda_{ijk}, {\rm d}z(e_{ab}) \seki = u_b -u_a, \qquad\text{where}\quad (a,b) = (i,j), (j,k),(k,i),
 \end{gather*}
 and take $\tau$ such that for a triple $T_{\mathrm{left}}$ and $T_{\mathrm{right}}$ of vertices
 which make a triangle left to the edge~$e$ and a triangle right to the edge $e$,
 \begin{gather*}
 \tau (e) = \lambda_{T_{\mathrm{left}}} -\lambda_{T_{\mathrm{right}}}.
 \end{gather*}
 Then the $q$ defined from the $\tau$ satisfies
 \begin{gather*}
 \sum_{e \in E_v} q(e) = \Delta_{\cot}u.
 \end{gather*}
 The $q$ is holomorphic if and only if $u$ is harmonic with respect to $\Delta_{\cot}$.
\end{Remark}

\section[Weierstrass representation formula for a discrete harmonic surface]{Weierstrass representation formula\\ for a discrete harmonic surface}

Given a topological trivalent graph \[M=(V,E) \subset \C\]
 with a complex coordinate $z$ of $U \subset M \subset \C$.
Now we are ready to give a Weierstrass representation with a pair of a holomorphic function $g$ in $z$
and a holomorphic quadratic differential $q$ on $M$ associated with $g$.

A {\em holomorphic quadratic differential associated with} $g$ is an extended notion, i.e., a function in the form of
\begin{gather*}
 q(e)= \ii \overline{\tau}(e) {\rm d}g(e)
\end{gather*}
with a discrete $1$-form $\tau$ satisfying
\begin{gather*}
 \sum_{e \in E_v} q(e) =0,
 \qquad
 \sum_{e \in E_v} \tau(e) =0 \qquad \text{for all}\quad v \in V.
\end{gather*}
This notion is preserved when $g$ is replaced by its linear fractional transformation.

Let us define $F\colon V \to \C^3$ by
\begin{align*}
 {\rm d}F(e_{ab}) &= \frac{\ii q(e_{ab})}{{\rm d}g(e_{ab})}(1-g_ag_b, \ii(1+g_ag_b), g_a+g_b)\\
 &=-\overline{\tau}(e_{ab}) (1-g_ag_b, \ii(1+g_ag_b), g_a+g_b),\qquad e_{ab} \in E,
\end{align*}
which connect the vertices $v_a$ and $v_b$ in $V$. Then we have the following lemma.

\begin{Lemma}
 \begin{equation}
 \label{eq:lemma:6:1}
 \frac{{\rm d}F(e_{ab})}{q(e_{ab})}
 = c_{ab}
 \left(
 N_a \times N_b +\ii(N_b -N_a)
 \right),
 \end{equation}
 where
 \begin{gather*}
 c_{ab} = \frac{\big(1+|g_a|^2\big)\big(1+|g_b|^2\big)}{|g_b-g_a|^2},
 \end{gather*}
 and
 for
 \begin{equation}
 \label{psuedonormal}
 N = \frac{1}{1+|g|^2} \big(g+ \overline{g}, \ii(\overline{g}-g), |g|^2-1\big).
 \end{equation}
\end{Lemma}
\begin{proof}
 The statement is obtained by simple calculations.
 The first term of the right-hand side of (\ref{eq:lemma:6:1}) $N_a \times N_b$ is parallel to the vector with the three elements.
 The first one
 \begin{align*}
 \ii\big[ &(\overline{g}_a-g_a)\big(|g_b|^2-1\big) -(\overline{g}_b-g_b)\big(|g_a|^2-1\big)
  =
 \ii\big[(\overline{g}_a\overline{g}_b -1)(g_b-g_a) -(g_ag_b-1)(\overline{g}_b -\overline{g}_a),
 \end{align*}
 the second one follows
 \begin{align*}
 &(\overline{g}_b+g_b)\big(|g_a|^2-1\big) -(\overline{g}_a+g_a)\big(|g_b|^2-1\big)\big]
 =
 -(\overline{g}_a\overline{g}_b +1)(g_b-g_a) -(g_ag_b+1)(\overline{g}_b -\overline{g}_a),
 \end{align*}
 and the third one follows
 \begin{align*}
 \ii&[ (\overline{g}_a+g_a)((\overline{g}_b-g_b) -(\overline{g}_b+g_b)(\overline{g}_a-g_a)]=\ii[(g_a+g_b)( \overline{g}_b-\overline{g}_a)]
 -(\overline{g}_a+\overline{g}_b )(g_b-g_a)].
 \end{align*}
 Thus we obtain
 \begin{align*}
 N_a \times N_b ={}&
 \frac{-\ii }{\big(1+|g_a|^2\big)\big(1+|g_b|^2\big)}
 \left[
 (g_b-g_a)
 (
 1-\overline{g}_a\overline{g}_b,
 -\ii(1+\overline{g}_a\overline{g}_b),
 \overline{g}_a+\overline{g}
 )
 \right.
 \\
 &
 \left.
 - (\overline{g}_b-\overline{g}_a)
 (
 1-g_ag_b,
 \ii(1+g_ag_b),
 g_a+g_b
 )
 \right]
 \\
 ={}&
 \frac{1}{(1+|g_a|^2)(1+|g_b|^2)}
 \left[
 (g_b-g_a) \frac{\overline{{\rm d}F(e_{ab})}{\rm d}\overline{g}(e_{ab})}{\overline{q}}
 +(\overline{g}_b-\overline{g}_a) \frac{{\rm d}F(e_{ab}) {\rm d} g(e_{ab})}{q}
 \right]
 \\
 ={}&c_{ab}^{-1}\big[\,\overline{{\rm d}F}/\overline{q}+ {\rm d}F/q\big](e_{ab}),
 \end{align*}
 where
 \begin{gather*}
 c_{ab} = \frac{\big(1+|g_a|^2\big)\big(1+|g_b|^2\big)}{|g_b-g_a|^2}.
 \end{gather*}
 Similarly, we compute $N_b-N_a$ with three elements.
 The first one
 is
 \begin{align*}
 &\frac{\big[(g_b+\overline{g}_b) \big(1+|g_a|^2\big) - (g_a+\overline{g}_a) \big(1+|g_b|^2\big)\big]}{\big(1+|g_a|^2\big)\big(1+|g_b|^2\big)}
 =
 \frac{[(1-\overline{g}_a\overline{g}_b)(g_b-g_a) + (1-g_ag_b)(\overline{g}_b-\overline{g}_a)]}{\big(1+|g_a|^2\big)\big(1+|g_b|^2\big)},
 \end{align*}
 the second one follows
 \begin{align*}
 &\frac{\ii\big[(\overline{g}_b-g_b) \big(1+|g_a|^2\big) - (\overline{g}_a-g_a) \big(1+|g_b|^2\big)\big]}{\big(1+|g_a|^2\big)\big(1+|g_b|^2\big)}\\
 &\qquad=
 \frac{\ii[-(1+\overline{g}_a\overline{g}_b)(g_b-g_a) + (1+g_ag_b)(\overline{g}_b-\overline{g}_a)]}{\big(1+|g_a|^2\big)\big(1+|g_b|^2\big)},
 \end{align*}
 the third one follows
 \begin{align*}
 &\frac{\big[\big(|g_b|^2-1\big)\big(|g_a|^2+1\big) \!-\!\big(|g_a|^2-1\big)\big(|g_b|^2+1\big)\big]}{\big(1+|g_a|^2\big)\big(1+|g_b|^2\big)}
 =
 \frac{[(g_a+g_b) (\overline{g}_b-\overline{g}_a) + (\overline{g}_a+\overline{g}_b) (g_b-g_a) ]}{\big(1+|g_a|^2\big)\big(1+|g_b|^2\big)}.
 \end{align*}
 Namely, by putting
 \begin{gather*}
 c_{ab}(N_b-N_a) :=\ii \big(\overline{{\rm d}F}/\overline{q}- {\rm d}F/q\big)(e_{ab}),
 \end{gather*}
 we show
 \[
 c_{ab}[N_a \times N_b +\ii (N_b-N_a)]
 = \big(\overline{F}/\overline{q}+ F/q\big)- \big(\overline{{\rm d}F}/\overline{q}- {\rm d}F/q\big)(e_{ab})
 = 2{\rm d}F(e_{ab})/q(e_{ab}).\tag*{\qed}
 \]\renewcommand{\qed}{}
\end{proof}

Now we check the balancing condition (\ref{eq:balancing}).
Let $v_1$, $v_2$, and $v_3$ be the neighboring vertices of a vertex $v_0$,
and $e_{0a}$ be the oriented edge connecting $v_0$ to $v_a$ with $a=1,2,3$.

Note
 \begin{align*}
 {\rm d}F(e_{0a}) &= - \overline{\tau}(e_{0a}) (1-g_0g_a, \ii(1+g_0g_a), g_0+g_a) \\
 &=- \overline{\tau}(e_{0a})\big(1-g_0^2 -g_0(g_a-g_0), \ii\big(1+g_0^2 +g_0(g_a-g_0)\big), 2g_0 +g_a-g_0\big) \\
 &=- \overline{\tau}(e_{0a})\big(1-g_0^2, \ii(1+g_0^2), 2g_0\big)
 + \ii \overline{\tau}(e_{0a})(g_a-g_0)g_0(\ii, -1, -\ii)\\
 &= - \overline{\tau}(e_{0a})\big(1+|g_0|^2\big)N
 +q(e_{0a})g_0(\ii, -1, -\ii).
 \end{align*}
The balancing condition
\begin{gather*}
 \sum_{a=1,2,3} \Re({\rm d}F(e_{oa})) = 0
\end{gather*}
follows from
\begin{gather*}
 \sum_{a=1,2,3}q(e_{0a})=0, \qquad \sum_{a=1,2,3}\tau(e_{0a}) =0.
\end{gather*}

\begin{Theorem}[Weierstrass representation formula for a discrete harmonic surface]
 \label{claim:weierstrass}
 Given a~trivalent graph $M=(M,V,E)$ and its complex coordinate $z \colon U \subset M \to \C$,
 with a planar part $U$ of $M$.
 For simplicity, we consider $M \subset \C$.
 For a holomorphic function $g$ in $z$ on $\C$ and a~holomorphic quadratic differential $q$ on a $M$ associate with $g$,
 let us define a map $X\colon M \to \R^3$
 \begin{gather*}
 X = \int \Re{\rm d}F,
 \end{gather*}
 with
 \begin{align}
 {\rm d}F(e_{ab}) &= \frac{\ii q(e_{ab})}{{\rm d}g(e_{ab})}(1-g_ag_b, \ii(1+g_ag_b), g_a+g_b)\nonumber\\
 &=-\overline{\tau}(e_{ab}) (1-g_ag_b, \ii(1+g_ag_b), g_a+g_b),\qquad e_{ab} \in E, \label{eq:weierstrass:1}
 \end{align}
which connect the vertices $v_a$ and $v_b$ in $V$. It defines a discrete harmonic surface in~$\R^3$.
 The above equation is called Weierstrass representation formula for a discrete harmonic surface.
\end{Theorem}

It should be noted the vector field $N$ defined by (\ref{psuedonormal}) is not a normal vector field unless $q$ is real valued.
We call it {\em pseudo normal} to the discrete harmonic surface.
The constructed surfaces are not always
closed globally. It would be interesting to study monodromies of the surfaces.
However, we may construct discrete Enneper-like harmonic surfaces (see Section~\ref{sec:numerics}).

\section[Convergence of discrete harmonic surfaces to a minimal surface]{Convergence of discrete harmonic surfaces\\ to a minimal surface}

In this section, we construct a sequence of discrete harmonic surfaces expressed with
the Weierstrass representation formula (\ref{eq:weierstrass:1}) with two holomorphic data $g$ and $q$,
and show its smooth convergence to a classical minimal surface with its Weierstrass date $g$ and $q$.
Key to the proof is to construct such a sequence with the ``common'' $g$ and $q$ on $\C$.

To illustrate our idea, let us consider a sequence of discrete harmonic surfaces $X_{\lambda_n} \colon M_{\lambda_n} \to \R^3$ via
the Weierstrass representation formula with common holomorphic function $g$ and holomorphic quadratic differential $q$ associated with $g$,
where $M_{\lambda_n}$ is a regular hexagonal lattice on $\C$ with edge length $\lambda_n > 0$ obtained by homothety
with multiplier $\lambda_n$ of the initial one.
Because the Weierstrass representation formula is a discrete approximation of the classical one, it is straightforward to show the following.

\begin{Theorem} \label{homothety}
Let us denote the discrete harmonic surface $X_{\lambda_n}$ constructed as above by~$X_n$.
The sequence $\{X_n\}$ converges to a $($classical$)$ closed minimal surface
$X_{\infty}\colon M_{\infty} \to \R^3$ when $\lambda_n \to 0$, and the mean curvature and the Gauss curvature of $X_n$ smoothly converge to that of the limit surface.
\end{Theorem}

\begin{proof}It is easy to see the sequence is a Cauchy sequence in the Hausdorff topology and have a minimal surface
$X_{\infty}\colon M_{\infty} \to \R^3$ as its limit which is defined by the classical Weierstrass representation formula with $g$ and $q$.

Due to
\begin{gather*}
 \nai N_a, {\rm d}F(e_{ab}) \seki = \overline{\tau}(e_{ab})(z_a-z_b) \qquad ( \text{for two adjacency vertices $v_a$ and $v_b$}),
\end{gather*}
the pseudo Gauss map $N$ converges to the Gauss map (unit normal vector field) of a limit minimal surface
because the distance between two adjacency vertices in $\C$ goes to zero as $n$ goes to the infinity.
Now we recall the definitions of the first fundamental form and the second fundamental form
for a discrete surface $X_n$ and a unit vector field $N$ (not necessarily the normal vector) are
difference analog of the classical one,
they converge to that of $X_{\infty}$ with the unit normal vector $N$,
and therefore the mean curvature and the Gauss curvature converges to that of the mean curvature and the Gauss curvature of the limit surface.
\end{proof}

Now let us construct a sequence of trivalent discrete harmonic surfaces $X_n\colon M_n \to \R^3 $ with~$M_0 =M$ given by the Weierstrass representation formula with common $g$ and $q$, and its convergence.
For the purpose, we have $M=(M,V,E)$ a planar trivalent graph with a~complex coordinate $z \in \C$. For simplicity, we consider $M$ embedded in $\C$.

For a given discrete harmonic surface $X\colon M \to \R^3$,
we have a sequence of topological discrete surfaces $M_n$ by the Goldberg--Coxeter subdivisions in $\C$.
In this section, we take twice of this process to construct $M_n$ from $M_{n+1}$ so that it looks like a homothety of multiplier $1/4$
in each subdivision step and all faces (polygonals) in $M_n$ become the homothetic polygons multiplied with $1/4$ in $M_{n+1}$.

Given a holomorphic function $g$ on $\C$, a holomorphic quadratic differential $q_n$ of each $M_n$ is uniquely determined
by $g$ up to scalar multiplicity.
Because the Goldberg--Coxeter subdivisions are homothety and $M_n$ converges to $\C$ in the Hausdorff topology
(i.e., $\bigcup M_n$ is dense in $\C$),
we can construct $q_n$ consistently and a differential form $q$ on $\C$ so that $q_n$ is a restriction of $q$.
The classical holomorphicity of $q$ follows from the discrete holomorphicity.

When we take the Goldberg--Coxeter subdivision,
then $\lambda_n = 4^{-n} \lambda$ and then by similar argument in Theorem~\ref{homothety}, we have the statement.

\begin{Corollary} \label{claim:homothety:cor}
 For a given discrete harmonic surface with holomorphic data $g$ and $q$,
 we construct a sequence of discrete harmonic surfaces which smoothly converges to a classical minimal surface
 with the holomorphic data~$g$ and~$q$.
\end{Corollary}

\section{Trivalent Enneper surface as numerical examples}\label{sec:numerics}

Given a regular hexagonal lattice in $\C =(\C, z)$,
the Weierstrass representation with a pair of a~holomorphic function $g$ on $\C$ and a~holomorphic quadratic differential $q$ defined on $E$ associate with $g$ is given by Theorem \ref{claim:weierstrass}.

In this section, as an application,
we study the convergence of the sequence of discrete surfaces~$X_n$ which are obtained
by subdivision of a given discrete surface $X_0$ with the Weierstrass data $q$ and $g$.
\begin{Lemma}
 \label{claim:lemma:8:1}
 Let $X\colon  (M,V,E) \to \R^3$ be a trivalent graph in $\R^3$,
 The holomorphic quadratic form $q$ with respect to $g$ is determined uniquely up scalar multiple.
\end{Lemma}

\begin{proof}
 Take a vertex $v_0 \in V$ and its nearest neighboring vertices $v_1$, $v_2$, and $v_3$.
 Their image in $\R^3$ are denoted by $x_0 =X(v_0)$ and $x_a =X(v_a)$ with $a=1,2,3$.
 Let $q_a= q(e_{0a})$, and~${c_a= g(x_a)-g(x_0)}$.
 Then the conditions for $q$ to be a holomorphic quadratic differentials is given as two linear equations:
 \begin{gather*}
 q_1+q_2+q_3=0,\qquad
 \frac{q_1}{c_1} + \frac{q_2}{c_2} + \frac{q_3}{c_3}=0.
 \end{gather*}
 The solution to them can be expressed
 \[
 q_1= \nu c_1(c_3-c_2),
 \qquad
 q_2= \nu c_2(c_1-c_3),
 \qquad
 q_3= \nu c_3(c_2-c_1),
 \]
 with an arbitrary scalar $\nu \in \C$.
\end{proof}

Now we compute {\em trivalent Enneper surface} and their convergence.
Let $X = (V, E)$ be a~regular hexagonal lattice whose edge length $\lambda > 0$ in $B_R(0) \subset \C$.
and a holomorphic function $g$ be $g(z) = z$.
Then by Lemma \ref{claim:lemma:8:1}, we may calculate the holomorphic quadratic differential $q$ as
 \begin{align}
 q_1 &= \nu (g(z_1) - g(z))(g(z_2) - g(z_3)), \qquad
 q_2 = \nu (g(z_2) - g(z))(g(z_3) - g(z_1)), \nonumber\\
 q_3 &= \nu (g(z_3) - g(z))(g(z_1) - g(z_2)),  \label{eq:numerical:q}
 \end{align}
for an arbitrary scalar $\nu \in \C$.
By using the Weierstrass formula (\ref{eq:weierstrass}) in Theorem \ref{claim:weierstrass}, we obtain
the local formula
\begin{equation}
 \label{eq:enneper}
 \begin{aligned}
 F(z_a)
 &=
 F(z)
 +
 \frac{\ii q_a}{g(z_a) - g(z)}
 \begin{bmatrix}
 1-g(z_a)g(z)
 \\
 \ii(1+g(z_a)g(z))\\
 g(z_a) + g(z)
 \end{bmatrix}
 \in \C^3.
 \end{aligned}
\end{equation}
Taking $F(z_0) = w_0 \in \C^3$ for a fixed vertex $z_0 \in V \subset \C$ and a fixed $w_0 \in \C^3$.
we may calculate~$F(z)$ using \eqref{eq:enneper} along a path from $z_0$ to $z \in V \subset \C$.
By defining $X(z) = \Re(F(z))$,
we call the surface $X$ a {\em trivalent Enneper surface} (see Figure \ref{fig:enneper}).

For the well-definedness of $X$,
we may show that the formula \eqref{eq:enneper} is closed along any closed path in the hexagonal lattice.
\begin{Lemma}
 For any $z_1, \ldots, z_5$, and $z_6$ be a closed path in a regular hexagonal lattice in $\C$,
 the formula {\rm \eqref{eq:enneper}} is closed along the path.
\end{Lemma}

\begin{figure}[t]
 \centering
 \includegraphics[bb=0 0 138 139,width=138pt]{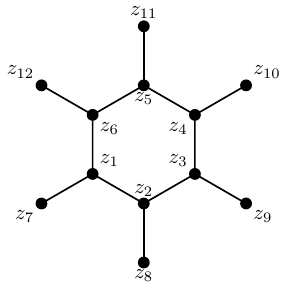}
 \caption{} \label{fig:hex_closed}
\end{figure}

\begin{proof}
 Let coordinates of $z_1, \ldots, z_6$ and $z_7, \ldots, z_{12}$ be
 \begin{align*}
 z_j = \lambda {\rm e}^{(j-1)\pi \ii/3 } + z_0,
 \qquad j = 1, \ldots, 6,
 \qquad
 z_k = 2 \lambda {\rm e}^{(k-7) \pi\ii/3} + z_0,
 \qquad
 k = 7, \ldots, 12,
 \end{align*}
 see Figure~\ref{fig:hex_closed}.
 By \eqref{eq:numerical:q} and \eqref{eq:enneper}, we obtain
 \begin{align}
 {\rm d}F(e_{j, j+1})
 &=
 F(z_{j+1})
 -
 F(z_j)
 =
 \ii (z_{j+6} - z_{j-1})
 \begin{bmatrix}
 1- z_j z_{j+1}, \\
 \ii(1+ z_j z_{j+1}) \\
 z_j + z_{j+1}
 \end{bmatrix}, \label{eq:enneper:2}
 \end{align}
 for $j = 1, \ldots, 6$,
 (in case of $j = 1$, set $j-1$ to be $6$).
 Substituting explicit complex coordinates into (\ref{eq:enneper:2}),
 we obtain
 \begin{align*} 
 {\rm d}F(e_{j, j+1})
 &=
 \ii \lambda \big(2 {\rm e}^{(j-1) \pi \ii/3} - {\rm e}^{(j-2) \pi \ii/3}\big)
 \begin{bmatrix}
 1-\lambda^2 \big({\rm e}^{(j-1) \pi \ii/3} + z_0\big)\big({\rm e}^{j \pi \ii/3} + z_0\big)\vspace{1mm}\\
 \ii\big(1+ \lambda^2 ({\rm e}^{(j-1) \pi \ii/3} + z_0\big)\big({\rm e}^{j \pi \ii/3} + z_0\big)\big) \vspace{1mm}\\
 \lambda \big({\rm e}^{(j-1) \pi \ii/3} + {\rm e}^{j \pi \ii/3} + 2 z_0\big)\\
 \end{bmatrix}
 \\
 &=
 \sqrt{3} {\rm e}^{(j+1)\pi\ii/3}
 \begin{bmatrix}
 \left(1-\lambda^2 \left(z_0+{\rm e}^{\ii \pi (j-1)/3}\right) \left(z_0+{\rm e}^{j \pi \ii/3}\right)\right)
 \vspace{1mm}\\
 \ii \left(1+\lambda^2 \left(z_0+{\rm e}^{\ii \pi (j-1)/3}\right) \left(z_0+{\rm e}^{j \pi \ii/3}\right)\right)
 \vspace{1mm}\\
 \lambda \left(2 z_0-\sqrt{3} {\rm e}^{-\pi \ii/6} {\rm e}^{j \pi \ii/3}\right)
 \end{bmatrix},
 \end{align*}
 and
 \[
 {\rm d}F(z_{1, 2}) + {\rm d}F(z_{2, 3}) + {\rm d}F(z_{3, 4}) + {\rm d}F(z_{4, 5}) + {\rm d}F(z_{5, 6}) + {\rm d}F(z_{6, 1}) = 0 \in \C^3.\tag*{\qed}
 \]\renewcommand{\qed}{}
\end{proof}

By taking Goldberg--Coxeter subdivision of the regular hexagonal lattice $M$,
we obtain a~regular hexagonal lattice $M_1$.
Constructing the same manner from $M_1$, we may obtain a trivalent Enneper surface $X_1$.
Continuing this procedure,
we obtain a sequence $\{X_j\}_{j=0}$ of trivalent Enneper surfaces (see Figure~\ref{fig:enneper}).

\begin{figure}[t] \centering
 \begin{tabular}{ccc}
 $M_0$
 &$M_1$
 &$M_2$
 \\
 \includegraphics[bb=0 0 150 145,width=100pt]{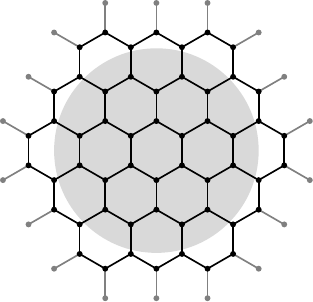}
 &\includegraphics[bb=0 0 126 137,width=100pt]{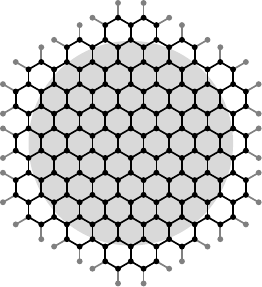}
 &\includegraphics[bb=0 0 113 113,width=100pt]{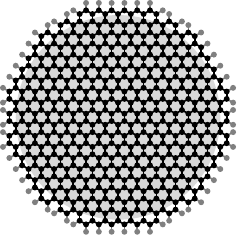}
 \end{tabular}
 \caption{Regular hexagonal lattices $M_k$ to construct trivalent Enneper surfaces $X_k$.
 Gray disks express~${B_{\sqrt{3}}(0) \subset \C}$.
 Note that classical Enneper surface from $B_{R}(0)$ have exactly self-intersections if~${R \ge \sqrt{3}}$.} \label{fig:lattice}
\end{figure}

\begin{figure}[t] \centering
 \begin{tabular}{cc}
 $M_0$ and $M_1$
 &$M_1$ and $M_2$
 \\
 \includegraphics[bb=0 0 149 143,width=100pt]{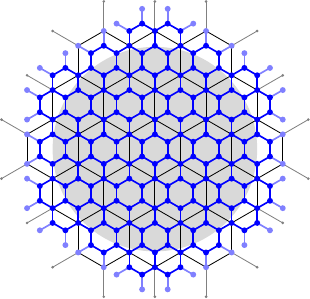}
 &\includegraphics[bb=0 0 124 136,width=100pt]{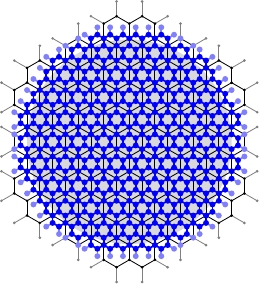}
 \end{tabular}
 \caption{Left: Goldberg--Coxeter construction from $M_0$ (black) to $M_1$ (blue).
 Right: Goldberg--Coxeter construction from $M_1$ (black) to $M_2$ (blue).} \label{fig:lattice1}
\end{figure}

We compute that the sequence $\{X_j\}$ converges a classical Enneper surfaces $X_{\mathrm{classical}}$,
Gauss maps and mean curvatures of $\{X_j\}$ converge the Gauss map and the mean curvatures of $X_{\mathrm{classical}}$, respectively
(see Figures~\ref{fig:lattice} and \ref{fig:lattice1}).
In fact, the Enneper surface $X_{\mathrm{classical}}$ appeared in the limit is
 \begin{align*}
 X= \sqrt{3}\nu \big(u^3/3 - u^2 v + u\big), \qquad
 Y= \sqrt{3}\nu \big(-v^3/3 + v^2 u - v\big), \qquad
 Z = \sqrt{3}\nu \big(u^2 - v^2\big), \end{align*}
where $\nu$ is the scalar used in \eqref{eq:numerical:q}.
We show numerical results of convergence
 \[
 X_k(z_j) - X_{\mathrm{classical}}(z_j),
 \qquad
 N_k(z_j) - N(z_j),
 \qquad
 H_k(z_j) - H(z_j)
 \]
in Figure~\ref{fig:converge}.

\begin{figure}[t]
 \centering
 \begin{tabular}{ccccc}
 $X_0$
 &$X_1$
 &$X_2$
 \\
 \includegraphics[bb=0 0 480 512,width=70pt]{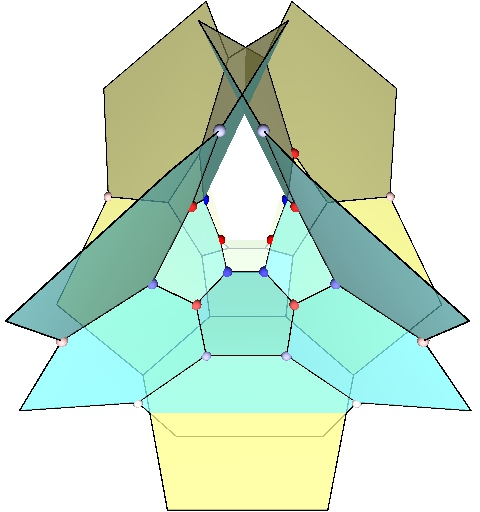}
 &\includegraphics[bb=0 0 576 496,width=70pt]{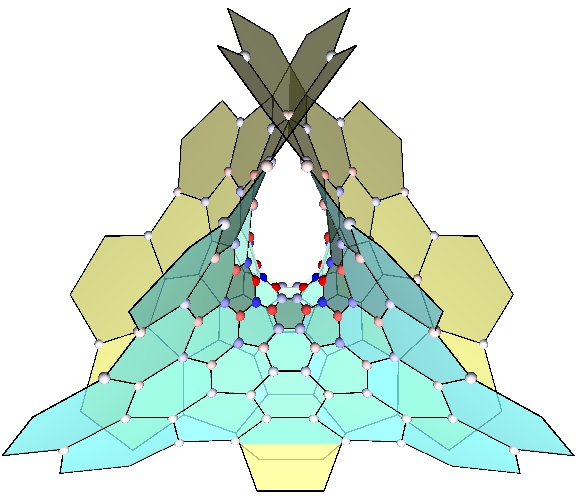}
 &\includegraphics[bb=0 0 400 400,width=70pt]{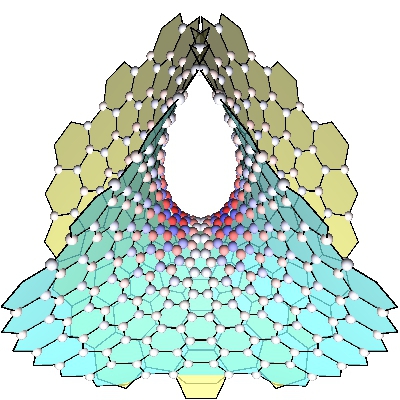}
 \\
 \includegraphics[bb=0 0 512 346,width=70pt]{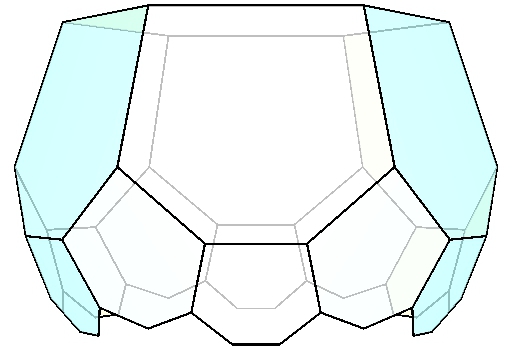}
 &\includegraphics[bb=0 0 512 416,width=70pt]{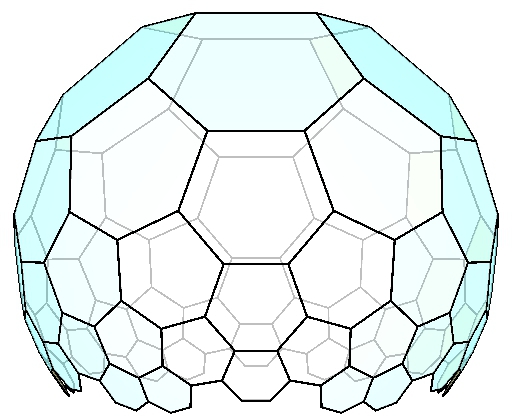}
 &\includegraphics[bb=0 0 512 416,width=70pt]{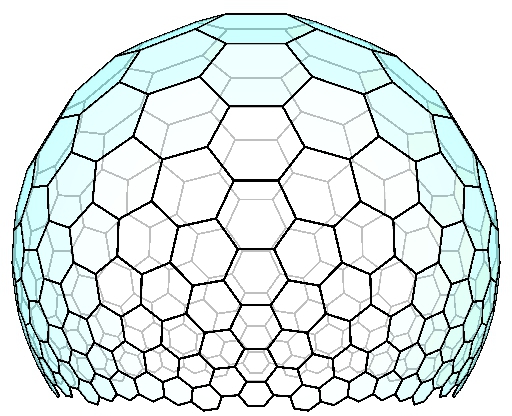}
 \\
 \mbox{}
 \\
 $X_3$
 &$X_4$
 &$X_5$
 \\
 \includegraphics[bb=0 0 416 381,width=70pt]{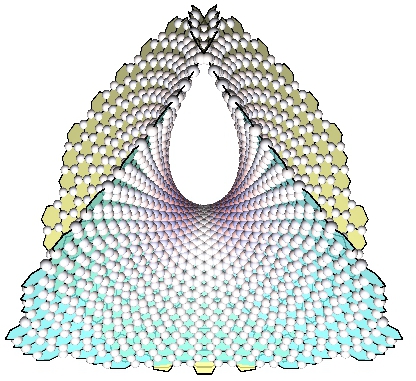}
 &\includegraphics[bb=0 0 384 360,width=70pt]{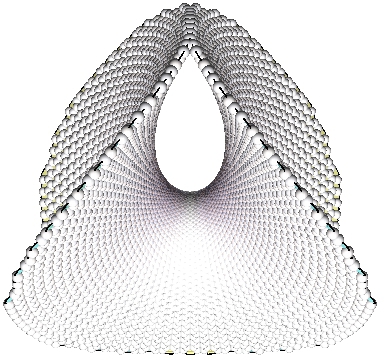}
 &\includegraphics[bb=0 0 368 352,width=70pt]{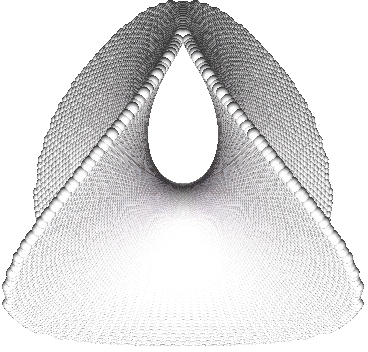}
 \\
 \includegraphics[bb=0 0 512 472,width=70pt]{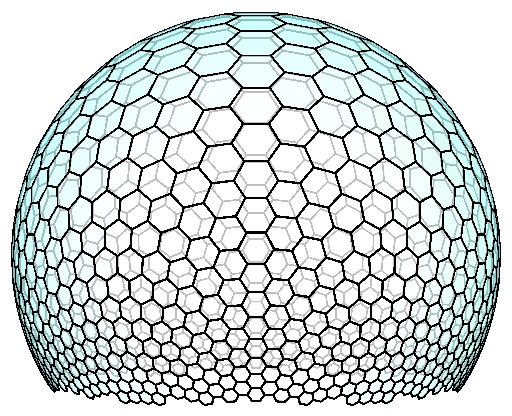}
 &\includegraphics[bb=0 0 512 472,width=70pt]{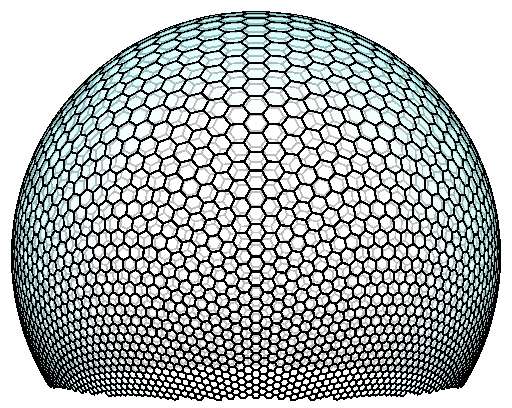}
 &\includegraphics[bb=0 0 512 472,width=70pt]{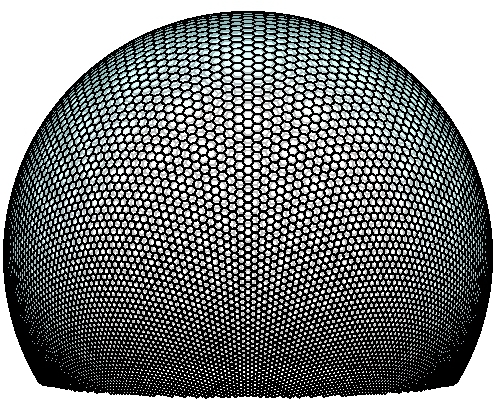}
 \end{tabular}
  \vspace{-1mm}

 \caption{Upper row:
 Trivalent Enneper surface $X_k$ (with $\nu = 1$ in holomorphic quadratic differential on $B_{\sqrt{3}}(0) \subset \C$),
 where $X_k$ is the $(k-1)$-times Goldberg--Coxeter subdivision of $X_1$.
 Vertex coloring are proportional to mean curvature (the larger the absolute value, the darker the color).
 Lower row: Their pseudo normal vector (pseudo Gauss map).} \label{fig:enneper}
\end{figure}
\begin{figure}[th!]
 \centering
 \includegraphics[bb=0 0 360 216,width=220pt]{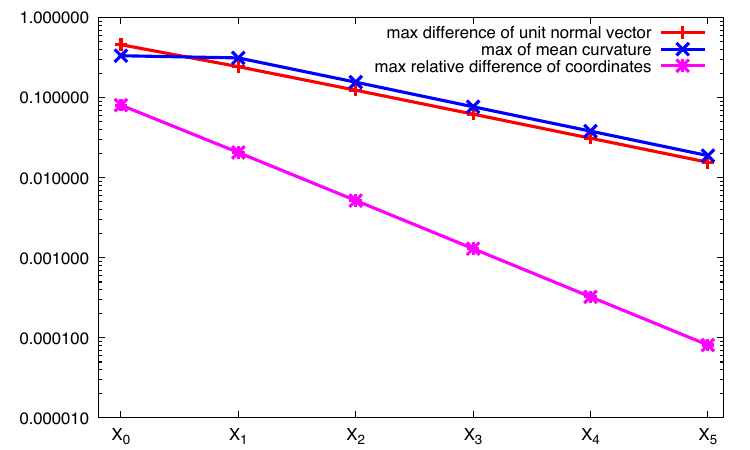}

 \vspace{-1mm}

 \caption{Blue: maximum values of difference of mean curvatures between $X_k$ and $X_{\mathrm{classical}}$.
 Red: maximum values of norm of difference between the Gauss map of $X_k$ and the Gauss map of $X_{\mathrm{classical}}$.
 Magenta: maximum difference of $X_k$ and $X_{\mathrm{classical}}$. }
 \label{fig:converge}
\end{figure}

\subsection*{Acknowledgements}
Motoko Kotani acknowledges the JSPS Grant-in-Aid for Scientific Research (B) under Grant No.~JP23H01072.
Hisashi Naito acknowledges the JSPS Grant-in-Aid for Scientific Research (C) under Grant No.~JP19K03488, No.~JP24K06710, and
the JSPS Grant-in-Aid for Scientific Research (B) under Grant No.~JP23H01072.
The authors would like to express their gratitude to the referee for its useful comments.

\pdfbookmark[1]{References}{ref}
\LastPageEnding

\end{document}